\documentclass[a4paper, 11pt]{article}

\usepackage{graphicx}
\usepackage[margin=4.1cm]{geometry}
\usepackage{bbm}
\usepackage{mathrsfs}
\usepackage{amsmath}
\usepackage{amssymb}
\usepackage{amsthm}
\usepackage[center]{caption}
\usepackage{enumerate}
\usepackage{verbatim}
\usepackage{authblk}
\usepackage{enumitem}
\usepackage{color}
\usepackage{calc}
\usepackage{longtable}
\usepackage{multirow}
\usepackage{hyperref}

\theoremstyle{plain}
\newtheorem{thm}{Theorem}[section]
\newtheorem{cor}[thm]{Corollary}
\newtheorem{lem}[thm]{Lemma}
\newtheorem{lemma}[thm]{Lemma}
\newtheorem{prop}[thm]{Proposition}

\theoremstyle{definition}

\renewcommand{\P}{\mathbb{P}}

\newcommand{\E}{\mathbb{E}}

\newcommand{\R}{\mathbb{R}}
\newcommand{\N}{\mathbb{N}}
\newcommand{\Px}{P^{\xi}}

\newcommand{\1}{1\hspace{-0.098cm}\mathrm{l}}

\newcommand{\eps}{\varepsilon}
\newcommand{\Z}{\mathbb{Z}}

\newcommand{\Prob}{\mathrm{Prob}}

\newcommand{\Om}{\Omega}

\newcommand{\p}{\P}

\DeclareMathOperator{\supp}{supp}

\newcommand{\be}[1]{\begin{equation}\label{#1}}
\newcommand{\ee}{\end{equation}}
\newcommand{\ra}{\rightarrow}
\newcommand{\bal}{\begin{aligned}}
\newcommand{\eal}{\end{aligned}}

\newcommand{\ssup}[1] {{\scriptscriptstyle{({#1}})}}

\newcommand{\cB}{\mathcal{B}}
\newcommand{\cC}{\mathcal{C}}

\newcommand{\cE}{\mathcal{E}}

\newcommand{\cO}{\mathcal{O}}

\newcommand{\Triple}{A}
\newcommand{\triple}{a}
\newcommand{\ttriple}{\tilde{a}}

\begin{document}

\begin{center}
{\Large \bf Scaling limit and ageing for branching random}\\[1mm]
{\Large \bf  walk in Pareto environment}\\[5mm]
\vspace{0.4cm}
\textsc{
Marcel Ortgiese\footnote{Institut f\"ur Mathematische Statistik,
Westf\"alische Wilhelms-Universit\"at M\"unster,
Einsteinstra\ss{}e 62,
48149 M\"unster,
Germany. } and 
Matthew I. Roberts\footnote{Department of Mathematical Sciences,
University of Bath,
Claverton Down,
Bath BA2 7AY,
United Kingdom.
Supported partially by an EPSRC postdoctoral fellowship (EP/K007440/1).}
}
\\[0.8cm]
{\small \today} 
\end{center}

\vspace{0.3cm}

\begin{abstract}\noindent 
We consider a branching random walk on the lattice, where the branching rates are given by an i.i.d.\ Pareto random potential.
We show that the system of particles, rescaled in an appropriate way, converges in distribution to a scaling limit that is interesting in its own right. We describe the limit object as a growing collection of ``lilypads'' built on a Poisson point process in $\R^d$. As an application of our main theorem, we show that the maximizer of the system displays the ageing property.
  \par\medskip

  
  \end{abstract}

\section{Introduction and main results}

\subsection{Introduction}

Consider a branching random walk in random environment defined on $\Z^d$, starting with a single particle at the origin. Given a collection $\xi = \{ \xi(z) \, : \, z  \in \Z^d\}$ of non-negative random variables, when at site $z$ each particle branches into two particles at rate $\xi(z)$. Besides this, each particle moves independently as a simple random walk in continuous time on $\Z^d$.

This model was introduced in~\cite{GM90}, and most of the analysis thus far has concentrated on the expected number of particles.
Fix a realisation of the environment $\xi$ and write
\[ u(z,t) = E^\xi [ \# \{ \mbox{particles at site } z \mbox{ at time } t  \} ] , \]
where the expectation $E^\xi$ is only over the branching and random walk mechansims and $\xi$ is kept fixed.
Then $u(z,t)$ solves the stochastic partial differential equation, known as the \emph{parabolic Anderson model} (PAM),
\[ \begin{aligned} \partial_t u(z,t) &  = \Delta u(z,t) + \xi(z) u(z,t) , & \quad \mbox{for  }z \in \Z^d, t \geq 0, \\
u(z,0) & = \1_{\{z = 0 \}} &\quad \mbox{for } z \in \Z^d .
\end{aligned}\]
Here, $\Delta$ is the discrete Laplacian defined for any function $f : \Z^d \ra \R$ as
\[ \Delta f(z) = \sum_{y \sim z} (f(y) - f(z)) , \quad z \in \Z^d, \]
where we write $y \sim z$ if $y$ is a neighbour of $z$ in $\Z^d$.

We are particularly interested in the case when the potential is Pareto distributed, i.e.~$\Prob(\xi(z)>x) = x^{-\alpha}$ for all $x\geq1$ and some $\alpha>0$. In this case, the evolution of the PAM is particularly well understood, including asymptotics for the total mass, one point localisation and a scaling limit: see~\cite{HMS08, KLMS09, MOS11,OR14}.

In general much less is known about the branching system itself (without taking expectations). Some of the earlier results include~\cite{ABMY00} and~\cite{GKS13}, who look at the asymptotics of the expectation (with respect to $\xi$) of higher moments of the number of particles. The real starting point for this article is our recent article~\cite{OR14}. We showed that---in the Pareto case---the hitting times of sites, the number of particles, and the support in an appropriately rescaled system are well described by a process defined purely in terms of the environment $\xi$ (that is, given $\xi$, the process is deterministic), which we called the lilypad model.

Our central aim in this article is to show that this lilypad process, and therefore the branching system itself, has a scaling limit. This limit object is entirely new, and interesting in its own right: it is neither deterministic, as for example in~\cite{CP07b} for another variant
of branching random walk in random environment, nor is it
 a stochastic (partial) differential equation. 
Rather the limit is a system of interacting and growing $L^1$ balls in $\R^d$, centred at the points of a Poisson point process. We call this the {\em Poisson lilypad model}, and to avoid confusion we will refer to the lilypad model from~\cite{OR14} as the {\em discrete lilypad model} from now on.

As an application of this characterization, we show that the dominant site in the branching process---that is, the site that has more particles than any other site---remains constant for long periods of time, in fact for periods that increase linearly as time increases. This phenomenon is known as {\em ageing}, and was demonstrated for the PAM in \cite{MOS11}.

\subsection{Definitions and notation}

Before we can state our results precisely, we need to develop some machinery. Throughout this article we write $|\cdot |$ for the $L^1$-norm on $\R^d$. $B(z,R) = \{ x \in \R^d : |x-z| < R \}$ denotes the open ball of radius $R$ about $z$ in $\R^d$, and $\cB(z,R) = \{x\in \R^d : |x-z|\leq R\}$ the closed ball. For any measure $\nu$, we write $\supp\nu$ for the (measure theoretic) support of $\nu$.

We take a collection of independent and identically distributed random variables $\{ \xi(z),\, z \in\Z^d\}$ satisfying
\[ \Prob ( \xi(z) > x ) = x^{- \alpha}  \quad \mbox{for all } x \geq 1 , \]
for a parameter $\alpha > 0$ and any $z \in \Z^d$. We will also assume that $\alpha > d$, which is known to be necessary for the total mass of the PAM to remain finite \cite{GM90}.

For a fixed environment $\xi$, we denote by $P_y^\xi$  the law of the branching simple random walk in continuous time with binary branching and branching rates $\{\xi(z)\, , \, z \in \Z^d\}$ started with a single particle at site $y$.
Finally, for any measurable set $F \subset \Om$, we define
\[ \p_y ( F \times \cdot) = \int_F P_y^\xi ( \cdot) \,\Prob(d \xi) . \]
If we start with a single particle at the origin, we omit the subscript $y$ and simply write
$P^\xi$ and $\P$ instead of $P_0^\xi$ and $\p_0$.

We define
$Y(z,t)$ to be the set of particles at the point $z$ at time $t$, and let $N(z,t) = \# Y(z,t)$.

We introduce a rescaling of time by a parameter $T > 0$, and then also rescale space and the potential. 
Setting  $q=\frac{d}{\alpha-d}$, the right scaling factors turn out to be
\[ a(T) = \left(\frac{T}{\log T}\right)^q \quad \mbox{and} \quad r(T) = \left(\frac{T}{\log T}\right)^{q+1} \]
for the potential and space respectively. We then define the rescaled lattice as
\[ L_T = \{ z \in \R^d \, : \, r(T) z \in \Z^d \}, \]
and for $z\in\R^d$, $R\geq 0$ define $L_T(z,R) = L_T\cap B(z,R)$. For $z \in L_T$, the rescaled potential is given by 
\[ \xi_T(z) = \frac{\xi(r(T) z)}{a(T)},  \]
and we set $\xi_T(z) = 0$ for $z \in \R^d \setminus L_T$.

\textbf{The branching system}

We are interested primarily in three functions:
\[H_T(z) = \inf\{t\geq0 : Y(r(T)z,tT)\neq \emptyset\}, \]
\[M_T(z,t) = \frac{1}{a(T)T}\log_+ N(r(T)z,tT), \]
and
\[S_T(t) = \{y\in\R^d : H_T(y)\leq t\}, \]
for $z \in L_T$, $t \geq 0$, which we extend to $z \in \R^d$ by linear interpolation.
We call these functions the (rescaled) hitting times, numbers of particles, and support, respectively, of the branching system.

\textbf{The scaling limit: the Poisson lilypad model}

In order to describe the limits of these functions as $T\to\infty$, we suppose that under $\P$ there is an independent Poisson point process $\Pi$ on $\R^d\times[0,\infty)$ with intensity measure $dz \otimes \alpha x^{-(\alpha+1)}dx$. We let $\Pi^\ssup{1}$ be the first marginal of $\Pi$, and write
\[\Pi = \sum_{i=1}^\infty \delta_{(z_i, \xi_\Pi(z_i))}\]
where $z_i$, $i=1,2,\ldots$ are the points in $\supp\Pi^\ssup{1}$.

We define, for $z\in\R^d$ and $t\geq0$,
\[ h(z) = \inf_{  y_1,y_2, \ldots \in \supp \Pi^\ssup{1}, y_n \ra 0  } \bigg\{ \sum_{j=1}^\infty q \frac{| y_{j+1} - y_j|}{\xi_\Pi(y_{j+1})} + q \frac{|y_1 - z|}{\xi_\Pi(y_1)} \bigg\}, \]
\[ m(z,t) = \sup_{y \in \supp\Pi^\ssup{1}} \big\{ \xi_\Pi(y) (t - h(y)) - q |y-z| \big\} \vee 0, \]
and
\[ s(t) = \{y\in\R^d : h(z)\leq t\}.\]
We recall that here and throughout $|\cdot|$ denotes the $L^1$-norm on $\R^d$.
We call these functions the hitting times, numbers of particles, and support, respectively, of the Poisson lilypad process.
We think of each site $y \in \supp \Pi^\ssup{1}$ as being home to a lilypad, which grows at speed $\xi_\Pi(y)/q$. 
However, these lilypads only begin to grow once they are touched by another lilypad. A simulation of the process can be seen at \url{http://tiny.cc/lilypads}.
We will see in Lemma~\ref{le:lilypad_nontrival} that these quantities are non-trivial, so that in particular the system does manage to
start growing from the origin, and does not explode in finite time.

\textbf{Topologies}

Write $C(A,B)$ for the set of continuous functions from $A$ to $B$. We use the following topologies:
\begin{itemize}
 \item For the hitting times: $\cC^d := C(\R^d , [0,\infty))$, equipped with the topology of uniform convergence on compacts, i.e.\  induced by the metric
 \[ d_{U} ( f, g) = \sum_{n \geq 1} 2^{-n}   \Big( \sup_{x \in [-n,n]^d} \{|f(x) - g(x)| \} \wedge 1\Big), \quad f, g \in  \cC^d. \]
 \item For the number of particles:
 \[\cC_0^{d+1} := \{f\in C(\R^d\times[0,\infty), [0,\infty)) : f(x,t)\to 0 \hbox{ as } x\to\infty\,\,\, \forall t\in[0,\infty)\},\] 
 equipped with the topology induced by the metric
 \[ d_{P} (f,g) = \sum_{n \geq 1} 2^{-n}   \Big( \sup_{x\in\R^d, t \in [0,n]} \{|f(x,t) - g(x,t)| \} \wedge 1\Big), \quad f, g \in  \cC_0^{d+1}. \]
 \item For the support: $\cC_F := C([0,\infty), F(\R^d))$, equipped with the topology induced by the metric
 \[ d_{F} ( f, g) = \sum_{n \geq 1} 2^{-n}   \Big( \sup_{t \in [0,n]} \{d_H(f(t),g(t)) \} \wedge 1\Big), \quad f, g \in  \cC_F, \]
 where $F(\R^d)$ is the space of non-empty compact subsets of $\R^d$ and $d_H$ is the Hausdorff distance on $F(\R^d)$.
\end{itemize}

Finally, we consider $(H_T, M_T, S_T)$ and $(h,m,s)$ as elements in the product space $\cC^{\ssup{ \times 3}} := \cC^d \times \cC^{d+1}_0 \times \cC_F$
equipped with the product topology, which is, for example, induced by the metric
\[ d^\ssup{\times 3} ((H, M, S), ( H', M', S')) = d_{U} ( H, H') + d_{P} ( M, M') + d_{F}( S, S') , \]
for any $(H, M, S), ( H', M', S') \in \cC^\ssup{\times 3}$.

\subsection{Main results}

Our main theorem states that the rescaled branching system (hitting times, number of particles and support) converges weakly to the Poisson lilypad model. For background on weak convergence, we refer to~\cite{Bil99, EthierKurtz86}.

\begin{thm}\label{thm:weak_conv}
The triple $(H_T, M_T, S_T)$ converges weakly in $\cC^\ssup{\times 3}$ as $T \ra \infty$ to $(h,m,s)$.
\end{thm}

As an application, we show that the maximal site in the branching system---that is, the site with the most particles at a given time---shows {\em ageing} behaviour. Denote by $Z^{\rm max}(t)$ this site: that is,
\[N(Z^{\rm max}(t),t) \geq N(z,t) \quad \forall z\in\Z^d;\]
in case of a tie choose the point with larger potential. Introduce the rescaled version
\[ W_T(t) := Z^{\rm max}(tT) / r(T). \]
Also let $w(t)$ be the maximizer in the Poisson lilypad model,
\[ m(w(t),t) \geq m(z,t) \quad \forall z\in\R^d;\]
again in the case of a tie we choose the site with larger potential (although we will show in Lemma \ref{le:Poisson_max_unique} that for any $t\geq0$ there is almost surely a unique maximizer for the Poisson lilypad model).

\begin{thm}\label{thm:ageing}
 \emph{Ageing.} For any $\theta > 0$,
 \[ \p ( Z^{\rm max}_{T} = Z^{\rm max}_{(1+\theta) T} ) = \p( W_T(1) = W_T(1+ \theta)) \ra \p ( w(1) = w(1+ \theta) ) . \]
\end{thm}

In the companion paper~\cite{onepoint}, we show that 
with high probability, the total mass of the  branching process is 
concentrated in a single point, so the theorem   really 
describes ageing, i.e.\ the temporal slow-down, of this maximizer.

The strategy of proof of Theorem~\ref{thm:weak_conv} relies on our previous result from~\cite{OR14}, which 
shows that the branching system is well described by a functional purely of
the environment, which we call the discrete lilypad model and recall in 
Section~\ref{ssn:discrete_lilypad}.
Then, our main task is to show that the discrete lilypad model converges
to the Poisson lilypad model that we described above. 
The underlying reason is that the rescaled environment converges to a Poisson process; see Section~\ref{ssn:point_pr}
for some background. 
The proof of Theorem~\ref{thm:weak_conv} is then an application of 
the continuous mapping theorem for a suitable continuous approximation 
of the lilypad models, which we describe in Section~\ref{sec:proof_scaling}. 
This approach allows us to avoid some of the technicalities involved with a more traditional 
approach of showing tightness combined with the convergence of finite dimensional distributions.
The proof of Theorem~\ref{thm:ageing} in Section~\ref{sec:proof_ageing} is then 
an application of the scaling limit.

Throughout the article, the ideas remain fairly simple, but there are many technicalities due to the highly sensitive nature of the model. For example, if one site of large potential is hit slightly earlier or later than it should be, the whole system could be affected dramatically. We have to keep track of several events that could, feasibly, occur; show that they have small probability; and show that if these events do not occur then the system behaves as we claim.

\subsection{The discrete lilypad model}\label{ssn:discrete_lilypad}

In~\cite{OR14}, we showed that the branching system is well-approximated
by certain functionals of the environment, which we will refer to as the \emph{discrete
lilypad model}.
For any site $z \in L_T$, we set
\[ h_T(z) = \inf_{ \substack{y_0,\ldots,y_n \in L_T:\\ y_0 = z, y_n = 0}} \Bigg( \sum_{j=1}^{n} q\frac{| y_{j-1} - y_{j}|}{\xi_T(y_{j})} \Bigg) .\]
We call $h_T(z)$ the first hitting time of $z$ in the discrete lilypad model. 
We think of each site $y$ as being home to a lilypad, which grows at speed $\xi_T(y)/q$. Note that $h_T(0)=0$.
For convenience, we interpolate $h_T$ linearly to define the values for $z \notin L_T$.
The rescaled number of particles in the discrete lilypad model is defined as
\[m_T(z,t) = \sup_{y \in L_T}\{\xi_T(y)(t-h_T(y)) - q|z-y|\}\vee 0. \] 
Also, we define the support of particles at time $t$ in the discrete lilypad model as
\[ s_T(t) = \{ z \in \R^d \, : \, h_T(z) \leq t \}. \]
We recall here the main result from~\cite{OR14}, which can be phrased as:

\begin{thm}[\cite{OR14}]\label{thm:discrete_lilypad}
For any $t_\infty > 0$, 
as $T \ra \infty$, 
\[ \bal 
\sup_{t \leq t_\infty} \sup_{z \in L_T} | M_T(z,t) - m_T(z,t) | \rightarrow 0 \quad \mbox{ in } \p\mbox{-probability} . 
\eal \]
Moreover, for any $R > 0$, as $T \rightarrow \infty$,
\[ \sup_{z \in L_T(0,R)} | H_T(z) - h_T(z) | \rightarrow 0 \quad \mbox{ in } \p\mbox{-probability} ,
 \]
and  for any $t_\infty >0$, as $T \ra \infty$,
\[ \sup_{t \leq t_\infty} d_H( S_T(t), s_T(t)) \rightarrow  0 \quad \mbox{in } \p\mbox{-probability}. \]
\end{thm}

We reiterate here the general idea behind this article: we know from Theorem \ref{thm:discrete_lilypad} that the branching system is well-approximated (with high probability) by the discrete lilypad model, which is a deterministic functional of the environment $\xi$. We can check that the distribution of $\xi$ (suitably rescaled) converges weakly to that of a Poisson point process; and this allows us to show that the discrete lilypad model converges weakly to the Poisson lilypad model.

\subsection{Background on point processes}\label{ssn:point_pr}

The proof of our main result, Theorem~\ref{thm:weak_conv}, is a 
consequence of the convergence of the rescaled environment to a Poisson process.
In this section we recall some of the standard definitions concerning point processes.

We consider the point process
\[ \Pi_T := \sum_{ z\in \Z^d} \delta_{( \frac{z}{r(T)},  \frac{\xi_T(z)}{a(T)}) } . \]
on $\R^d \times (0, \infty) $. A classical result in extreme value theory shows that
$\Pi_T$ converges in law to the Poisson point process $\Pi$ on $\R^d \times (0,\infty)$ with 
intensity measure 
\[ \pi( d (z,x) ) =  dz \otimes \, \frac{\alpha}{x^{\alpha+1}}\,d x . \]
In order to formalize this convergence we follow the basic 
setup  from~\cite{HMS08}, which is based on~\cite{Res08}.
	Let $E$ be a locally compact space with a countable basis and let $\cE$ denote the Borel-$\sigma$-algebra on $E$. 
A Radon measure is a Borel measure
that is locally finite. If in addition $\mu = \sum_{ i \geq 1} \delta_{x_i}$ for a countable 
collection of points $\{ x_i, i\ge 1 \} \subset E$, then $\mu$ is called a point measure.  We write $M_p(E)$ for the set of all point measures on $E$.
We equip the set of Radon measures $M_+(E)$
with the vague topology: i.e.\ $\mu_n \ra \mu$ vaguely, if for any 
continuous function $f : E \ra \R$ with compact support $\int f \mu_n \ra \int f \mu$.
Note that $M_p(E)$ is vaguely closed in $M_+(E)$ (cf.~\cite[Prop.~3.14]{Res08}).

In our case we set $E = \R^d \times (0,\infty]$, where the topology on $(0,\infty]$ is understood such that 
closed neighbourhoods of $\infty$ are compact. Note that $\Pi_T$ and $\Pi$ are elements of $M_p(E)$
for this choice.
Then the above convergence means that $\Pi_T \Rightarrow \Pi$
in the topology on $M_p(E)$ induced by vague convergence. This fact is a direct application of~\cite[Prop.~3.21]{Res08}
(where $\R^d$ replaces  $\R^+$ as the index set).

\section{Proof of the scaling limit}\label{sec:proof_scaling}

In this section we prove the main scaling limit, Theorem~\ref{thm:weak_conv}. 
By our previous result on the approximation via the discrete model, Theorem~\ref{thm:discrete_lilypad},
it suffices to show convergence of the discrete lilypad model. 
Our main strategy is to use the continuous mapping theorem to deduce the convergence
of $(h_T, m_T,s_T)$ from the convergence of the point process $\Pi_T$ to $\Pi$. 
Unfortunately, however, it is not clear that $(h_T,m_T,s_T)$ is a continuous function of the underlying
point process. Our way around this problem is to define an $\delta$-approximate lilypad model
for both the discrete space version and the Poisson model. 
By ignoring potential values less than $\delta$---and, later, restricting in space 
to $B(0,1/\delta)$---we obtain functionals that only depend on a finite set of points
and are therefore continuous.

We can treat both the discrete space and the Poisson case
in the same way. Thus, for $\nu = \Pi_T$, for some $T>0$,
or $\nu = \Pi$, we write $\nu\in M_p(E)$ as
\[ \nu = \sum_{i \geq 1}  \delta_{(z_i , \xi_\nu(z_i))}  ,\]
and write $\nu^\ssup{1} (\cdot) := \nu ( \,\,\cdot \times [0,\infty))$ for the first marginal of $\nu$. For $r>0$, we write $B_\nu(0,r) = \supp(\nu^\ssup{1})\cap B(0,r)$ and $\cB_\nu(0,r)=\supp(\nu^\ssup{1})\cap \cB(0,r)$. Where it is clear which point process we are referring to, we write $\xi(z)$ in place of $\xi_\nu(z)$ for conciseness. (Of course, we have already defined $\{\xi(z):z\in\Z^d\}$ to be a collection of i.i.d.~Pareto random variables; but since we already know from Theorem \ref{thm:discrete_lilypad} that the branching process is well approximated by the discrete lilypad model, which can be described via the point process $\Pi_T$, we no longer need this original meaning and $\xi(z)$ will always refer to $\xi_\nu(z)$ for some point process $\nu$.)

For a general point process $\nu$, we define the hitting times by setting $h_\nu(0)=0$ and, for $z\in\R^d\setminus\{0\}$,
\[h_\nu(z) =  \inf \Big\{ q \sum_{i=1}^\infty \frac{|y_i - y_{i-1}|}{\xi(y_i)}  \, : \, y_0 = z, y_i \in \supp\nu^\ssup{1} \ \forall i \geq 1\, , \, |y_i | \ra 0 \Big\} . \]
The number of particles is defined as
\[m_\nu(z,t) = \sup_{y \in \supp \nu^\ssup{1}} \Big\{ \xi(y) ( t - h_\nu(y) ) - q |y-z| \Big\} \vee 0 , \quad z \in \R^d, t \geq 0, \]
and the support is defined as
\[ s_\nu(t) = \{ z \in \R^d \,  : \, h_\nu(z) \leq t \} , \quad  t \geq 0 . \]

We also define the $\delta$-hitting times by setting
\[ h^\delta_\nu(z)  = \inf \Big\{ \sum_{j=1}^n q \frac{|y_{j-1} - y_j|}{\xi(y_j)} + q\frac{|y_n|}{\delta}  \, : \,n\in\N_0, y_0 = z \mbox{ and }y_1, \ldots, y_{n} \in \supp\nu^\ssup{1} \Big\}  \]
for any $z\in\R^d$ (note that we allow $n = 0$, in which case we do not insist on $y_n \in \supp(\nu^\ssup{1})$). Effectively, considering $h^\delta_\nu(z)$ rather than $h_\nu(z)$ gives all lilypads a ``minimum speed'' $\delta/q$, which helps in showing the continuity of the process as a function of the point measure $\nu$. In analogy with the definitions above, we also define the $\delta$-number of particles and the $\delta$-support via
\[  m_\nu^\delta(z,t) = \sup_{y \in \supp \nu^\ssup{1}} \Big\{ \xi(y) ( t - h_\nu^\delta(y) ) - q |y-z| \Big\} \vee 0 , \quad z \in \R^d, t \geq 0, \]
and
\[ s_\nu^\delta(t) = \{ z \in \R^d \,  : \, h_\nu^\delta(z) \leq t \} , \quad  t \geq 0 . \]
We write $(h_T^\delta, m_T^\delta, s_T^\delta) := (h_{\Pi_T}^\delta, m_{\Pi_T}^\delta, s_{\Pi_T}^\delta)$
and $(h^\delta, m^\delta, s^\delta) := (h_{\Pi}^\delta, m_{\Pi}^\delta, s_{\Pi}^\delta)$.

The main technical result of this section is the following proposition. 

\begin{prop}\label{prop:appr_works} For any $\eps > 0$,
\[\lim_{\delta \downarrow 0} \limsup_{T\to\infty}\p \Big( d^{(\times 3)} \big( ( h_T^\delta, m_T^\delta, s_T^\delta ), (h_T,m_T,s_T) \big) \geq \eps  \Big)  =  0, \]
and analogously for the Poisson point process
\[\lim_{\delta \downarrow 0}  \p \Big( d^{(\times 3)} \big( ( h^\delta, m^\delta, s^\delta ), (h,m,s) \big) \geq \eps  \Big)  =  0 . \]
\end{prop}

The remainder of this section is organised as follows. In Section~\ref{ssn:delta_hitting}, we give general
criteria on the point process $\nu$ that ensure that the $\delta$-hitting times 
approximate well the actual hitting times. 
Then in Section~\ref{ssn:delta_support} we show that this result can be transferred
to the number of particles and the support. In Section~\ref{ssn:delta_works} we show that
these general criteria are satisfied by the point processes $\Pi_T$ and $\Pi$, and
we prove Proposition~\ref{prop:appr_works}.
Finally, in Section~\ref{ssn:proof_scaling},
we show that the $\delta$-processes for $\Pi_T$ converge to the $\delta$-processes for $\Pi$, and we combine these results to show the statement
of the main scaling limit Theorem~\ref{thm:weak_conv}.

\subsection{The $\delta$-approximation of the hitting times} \label{ssn:delta_hitting}
 
We now state certain  assumptions on the point process $\nu$ under which $h_\nu^\delta$ and $h_\nu$ will be close when $\delta$ is small. Let $\gamma = \frac{d+\alpha}{2\alpha}$.
\begin{enumerate}[label={(A\arabic*)}]
 \item \label{asmp:pot_large} For all $R\geq R_0$, $\sup_{y \in B_\nu(0,R)} \xi(y) \leq q R^\gamma$.
 \item \label{asmp:small_pts} For all $r\leq r_0$, for all $k\in \N_0$, there exists $Z_k \in B_\nu(0,r2^{-k})$ such that $\xi(Z_k) \geq r^\gamma 2^{-k\gamma}$.
\end{enumerate}
We write \ref{asmp:pot_large}$_{R_0}$ and \ref{asmp:small_pts}$_{r_0}$ to emphasize the dependence of the conditions on the parameters. 

The main result in this subsection states that the hitting times are approximated well by the $\delta$-hitting times, provided $\nu$ satisfies the above conditions.

\begin{prop}\label{prop:del_app_general}
Suppose that $\nu$ satisfies~\ref{asmp:pot_large}$_{R_0}$ and~\ref{asmp:small_pts}$_{r_0}$. Then for any $\eps>0$, there exists $\delta>0$ (depending only on $\gamma$, $\eps$, $R_0$ and $r_0$) such that
\[h_\nu^\delta(z) \leq h_\nu(z) \leq h_\nu^\delta(z)+\eps \quad \forall z\in \R^d.\]
\end{prop}

We will also need the following two simple lemmas, which prove upper and lower bounds on the hitting times.

\begin{lem}\label{le:small_times}
Suppose that $\nu$ satisfies ~\ref{asmp:small_pts}$_{r_0}$. Then for any $r\leq r_0$,
\[ \max_{z \in B(0,r)} h_\nu(y) \leq \frac{ 4 q r^{1- \gamma}}{1 - 2^{\gamma-1}} \]
and moreover, for any $z\in\R^d$,
 \[  h_\nu(z) \leq \frac{ 4 q r_0^{1- \gamma}}{1 - 2^{\gamma-1}}  + q ( r_0^{-\gamma} |z| + r_0^{1 - \gamma} ). \]
\end{lem}

\begin{lem}\label{le:large_hitting}
Suppose that $\nu$ satisfies \ref{asmp:pot_large}$_{R_0}$. Then for any $R\geq R_0$ and any $\delta>0$,
\[ \inf_{y \not\in B(0,R) } h^\delta_\nu(y) \geq \min\{R^{1-\gamma}, qR/\delta\}.\]
\end{lem}

The lemmas lead easily to two useful corollaries.

\begin{cor}\label{cor:restrict}
Suppose that $\nu$ satisfies~\ref{asmp:pot_large}$_{R_0}$. Then for any $z\in\R^d$ and any $\delta>0$, there exists $R>0$ (depending only on $\gamma$, $R_0$ and $\delta$) such that the infimum in the definition of $h^\delta_\nu(z)$ can be restricted to points $y_1,\ldots,y_n\in B_\nu(0,R)$.
\end{cor}

\begin{cor}\label{cor:nontriv}
Suppose that $\nu$ satisfies~\ref{asmp:pot_large}$_{R_0}$ and~\ref{asmp:small_pts}$_{r_0}$ for some $R_0$ and $r_0$. Then for all $z\in\R^d\setminus\{0\}$ and all $\delta>0$, we have
\[0<h_\nu^\delta(z)\leq h_\nu(z)<\infty.\]
\end{cor}

We delay the proofs of the lemmas and corollaries for a moment to concentrate on Proposition \ref{prop:del_app_general}.

\begin{proof}[Proof of Proposition \ref{prop:del_app_general}]
The fact that $h_\nu^\delta(z) \leq h_\nu(z)$ for all $z\in\R^d$ follows immediately from the definitions, so we aim to prove that $h_\nu(z) \leq h_\nu^\delta(z)+\eps$.

Since $\gamma<1$ we may choose $\delta>0$ small enough so that
\begin{equation}\label{eq:del1}
4q\delta^{1-\gamma}\Big(\frac{1}{1-2^{\gamma-1}}\Big) \leq \eps,
\end{equation}
\begin{equation}\label{eq:del2}
(\delta/4)^{\gamma} \geq 2\delta \quad \mbox{ and } \quad \delta\leq r_0.
\end{equation}

By Corollary \ref{cor:restrict}, there exists $R>0$ such that the infimum in the definition of $h_\nu^\delta(z)$ is taken over points $y_1,\ldots,y_n\in B_\nu(0,R)$; we also note from the definition that necessarily $\xi(y_i)\geq\delta$ for each $i=1,\ldots,n$. Since the set $B(0,R)\times[\delta,\infty)$ is relatively compact in $E$, and $\nu$ is a Radon measure, there are only finitely many such points. Thus the infimum is actually a minimum, and we can find points $y_0=z,y_1,\ldots,y_n$ such that
\begin{equation}\label{eq:2212-2}h_\nu^\delta(z) = \sum_{i=1}^{n} q \frac{|y_{i-1}- y_i|}{\xi(y_i)} + q \frac{|y_n|}{\delta}.\end{equation}
Note from the definition of $h_\nu$ that
\[h_\nu(z) \leq h_\nu(y_n) + q\sum_{i=1}^{n} \frac{|y_i-y_{i-1}|}{\xi(y_i)} \leq h_\nu(y_n) + h_\nu^\delta(z),\]
so it remains to prove that $h_\nu(y_n)\leq \eps$.

By Lemma \ref{le:small_times} and the fact that $\delta\leq r_0$, together with \eqref{eq:del1}, we have
\[ \max_{y\in B(0,\delta)} h_\nu(y) \leq \eps.\]
Thus it suffices to prove that $|y_n|<\delta$.

By \ref{asmp:small_pts}$_{r_0}$ with $r=\delta\leq r_0$ and $k=2$, we can choose $Z\in B(0,\delta/4)$ such that $\xi(Z)\geq (\delta/4)^{\gamma} \geq 2\delta$ by~\eqref{eq:del2}. Suppose that $|y_n|\geq\delta$. Then
\[ \frac{|Z|}{\delta} + \frac{|Z - y_n|}{\xi(Z)} 
\leq \frac{|Z|}{\delta} + \frac{|Z - y_n|}{2 \delta}  \leq
\frac{|Z|}{\delta} + \frac{|Z|}{2\delta} + \frac{|y_n|}{2 \delta}
\leq \frac{3}{2} \frac{\delta/4}{\delta} + \frac{|y_n|}{2 \delta}
< \frac{|y_n|}{\delta} . 
\]
Thus by including $Z$ in the approximating sequence we get a smaller value of $h_\nu^\delta(z)$ than~\eqref{eq:2212-2}, contradicting the optimality of the sequence $y_0,\ldots,y_n$. We deduce that $|y_n| < \delta$ as required.
\end{proof}

We now proceed with the proofs of the lemmas. Lemma \ref{le:small_times} follows easily from the assumption~\ref{asmp:small_pts}$_{r_0}$:

\begin{proof}[Proof of Lemma \ref{le:small_times}]
 Fix $r\leq r_0$ and let $Z_k$, $k \geq 0$, be as in~\ref{asmp:small_pts}$_{r_0}$. Then by definition, for any $z \in B(0,r)$, 
 we have 
 \[ h_\nu(z) \leq q \frac{|z - Z_0|}{\xi(Z_0)} + q \sum_{j = 1}^\infty \frac{|Z_{j-1} - Z_{j}|}{\xi(Z_{j})} \leq 2q \frac{r}{r^\gamma} + q \sum_{j=1}^\infty  \frac{ 2 r 2^{-(j-1)} }{ r^\gamma 2^{- \gamma j }} \leq 4 q r^{1- \gamma} \frac{1}{1 - 2^{\gamma-1}} .\]

For the second claim, taking $r=r_0$ in the above, we have that for any $z$,
\[ h_\nu(z) \leq h_\nu(Z_0)+ q\frac{|z- Z_0|}{\xi(Z_0)} 
\leq \frac{ 4 q r_0^{1- \gamma}}{1 - 2^{\gamma-1}}  + \frac{q (|z| + r_0)}{r_0^\gamma}.\qedhere \]
\end{proof}

Lemma \ref{le:large_hitting} is slightly more fiddly.

\begin{proof}[Proof of Lemma \ref{le:large_hitting}]
It is easy to see from the definition that $z\mapsto h^\delta_\nu(z)$ is continuous. Therefore there exists a point $\tilde z \in \partial B(0,R)=\{z\in\R^d : |z|=R\}$ that minimizes $h^\delta_\nu$, i.e.\ 
\[h^\delta_\nu(\tilde z) = \inf_{ y \in \partial B(0,R)} h_\nu^\delta(y).\]

We claim that $h^\delta_\nu(\tilde z) = \inf_{|y| \geq R} h_\nu^\delta(y)$. Indeed, suppose there exists $y\not\in B(0,R)$ with $h_\nu^\delta(y) < h_\nu^\delta(\tilde z)$. Then we can choose $y_0=z, y_1,\ldots, y_n$ with
\[q\sum_{j=1}^n \frac{|y_j-y_{j-1}|}{\xi(y_j)} + q\frac{|y_n|}{\delta} < h_\nu^\delta(\tilde z).\]
We may assume without loss of generality that $y_1\in B(0,R)$ (since clearly $h_\nu^\delta(y_1) < h_\nu^\delta(\tilde z)$, so we can otherwise use $y_1$ in place of $y$). Therefore there exists $a\in(0,1)$ such that $\tilde y:=y_1 + a(y-y_1)\in \partial B(0,R)$. Then
\begin{align*}
h_\nu^\delta(\tilde y) \leq q\frac{|y_1-\tilde y|}{\xi(y_1)} + q\sum_{j=2}^n \frac{|y_j-y_{j-1}|}{\xi(y_j)} + q\frac{|y_n|}{\delta} &= qa\frac{|y_1- y|}{\xi(y_1)} + q\sum_{j=2}^n \frac{|y_j-y_{j-1}|}{\xi(y_j)} + q\frac{|y_n|}{\delta}\\
& < q\sum_{j=1}^n \frac{|y_j-y_{j-1}|}{\xi(y_j)} + q\frac{|y_n|}{\delta} = h_\nu^\delta(\tilde z),
\end{align*}
contradicting the choice of $\tilde z$. Therefore the claim holds.

Since $h_\nu^\delta(y) > h_\nu^\delta(\tilde z)$ for all $y\not\in B(0,R)$, we see that the infimum in the definition of $h_\nu^\delta(\tilde z)$ can be restricted to points within $B(0,R)$: that is,
\[h_\nu^\delta (\tilde z) = \inf\Big\{q \sum_{j=1}^n \frac{|y_{j-1} - y_j|}{\xi(y_j)} + q \frac{|y_n|}{\delta} \, : \,n\in\N_0, y_0 = x \mbox{ and }y_1, \ldots, y_{n} \in B_\nu(0,R) \Big\}.\]
In particular,
\[h_\nu^\delta(\tilde z) \geq \min\Big\{ \frac{qR}{\max_{y \in B_\nu(0,R)} \xi(y)} , \frac{qR}{\delta} \Big\},\]
and therefore by~\ref{asmp:pot_large}$_{R_0}$, if $R\geq R_0$ then
\[\inf_{y\not\in B(0,R)} h_\nu^\delta(y) \geq \min\{ R^{1-\gamma}, qR/\delta\}.\qedhere\]
\end{proof}

\begin{proof}[Proof of Corollary \ref{cor:restrict}]
Fix $z\in\R^d$. By Lemma \ref{le:large_hitting}, we can choose $R$ large enough such that
\[\inf_{y\not\in B(0,R)} h_\nu^\delta(y) > h_\nu^\delta(z).\]
Therefore the infimum in the definition of $h_\nu^\delta(z)$ can be restricted to points within $B(0,R)$.
\end{proof}

\begin{proof}[Proof of Corollary \ref{cor:nontriv}]
Take any $z\in\R^d\setminus\{0\}$ and $\delta>0$. By Corollary \ref{cor:restrict}, there exists $R>0$ such that the infimum in the definition of $h^\delta_\nu(z)$ can be restricted to points within $B(0,R)$, so
\[h_\nu^\delta(z) \geq \min\Big\{\frac{q|z|}{\max_{y\in B_\nu(0,R)}\xi(y)},\frac{q|z|}{\delta}\Big\} > 0.\]
The fact that $h_\nu^\delta(z)\leq h_\nu(z)$ follows directly from the definitions; and $h_\nu(z)<\infty$ by Lemma \ref{le:small_times}.
\end{proof}

\subsection{The $\delta$-approximation of the support and number of particles}\label{ssn:delta_support}

We recall that
\[  m_\nu(z,t) = \sup_{y \in \supp \nu^\ssup{1}} \Big\{ \xi(y) ( t - h_\nu(y) ) - q |y-z| \Big\} \vee 0 , \quad z \in \R^d, t \geq 0 \]
and
\[ s_\nu(t) = \{ z \in \R^d \,  : \, h_\nu(z) \leq t \} , \quad  t \geq 0, \]
and that $m_\nu^\delta(z,t)$ and $s^\delta_\nu(t)$ are defined similarly by replacing $h_\nu$ by $h_\nu^\delta$. 
In this subsection, we show that under~\ref{asmp:pot_large}$_{R_0}$ and~\ref{asmp:small_pts}$_{r_0}$, the $\delta$-approximations $m_\nu^\delta$ and $s_\nu^\delta$ are close to $m_\nu$ and $s_\nu$ respectively.

We start by showing that the growth of the support $s_\nu$ is well-controlled. This will be key to controlling the Hausdorff distance between $s_\nu$ and $s_\nu^\delta$.

\begin{lemma}\label{le:le61}
Suppose that $\nu$ satisfies~\ref{asmp:pot_large}$_{R_0}$. For any $\eps>0$ and any $t_0>0$, there exists $\eta \in (0,1)$ (depending only on $\gamma$, $\eps$, $t_0$ and $R_0$) such that
\[  s_\nu(t+\eta) \subseteq \bigcup_{y \in s_\nu(t)} B(y,\eps) \quad \forall t\leq t_0. \]
\end{lemma}

\begin{proof}
By Lemma \ref{le:large_hitting}, together with the fact that $h_\nu(z)\geq h_\nu^\delta(z)$ for all $z$, we can choose $R\geq R_0$ such that $h_\nu(y)>t_0+1$ for all $y\not\in B(0,R)$. Then set $\eta = \frac{\eps}{2R^\gamma}\wedge \frac12$.

Suppose that $z\in s_\nu(t+\eta)\setminus s_\nu(t)$; then $h_\nu(z)\in(t,t+\eta]$, so we can find $y_0=z,y_1,y_2\ldots\to0$ with $h_\nu(z)\leq q\sum_{i=1}^\infty \frac{|y_i-y_{i-1}|}{\xi(y_i)} \leq t+2\eta$. Since $h_\nu(y)>t_0+1$ for all $y\not\in B(0,R)$, we must have $y_1,y_2,\ldots \in B(0,R)$.

Choose $k$ such that $q\sum_{i=k+1}^\infty \frac{|y_i-y_{i-1}|}{\xi(y_i)} \leq t$ and $q\sum_{i=k}^\infty \frac{|y_i-y_{i-1}|}{\xi(y_i)} > t$. Then choose $a\in[0,1)$ such that
\[q\sum_{i=k+1}^\infty \frac{|y_i-y_{i-1}|}{\xi(y_i)} + aq\frac{|y_k-y_{k-1}|}{\xi(y_k)} = t.\]
Setting $\tilde y = y_k + a(y_k-y_{k-1})$, by the above we have $h_\nu(\tilde y)\leq t$, so $\tilde y\in s_\nu(t)$. On the other hand,
\begin{multline*}
q\sum_{i=k+1}^\infty \frac{|y_i-y_{i-1}|}{\xi(y_i)} + aq\frac{|y_k-y_{k-1}|}{\xi(y_k)}\\
= q\sum_{i=1}^\infty \frac{|y_i-y_{i-1}|}{\xi(y_i)} - (1-a)q\frac{|y_k-y_{k-1}|}{\xi(y_k)} - q\sum_{i=1}^{k-1} \frac{|y_i-y_{i-1}|}{\xi(y_i)},
\end{multline*}
so (since the left-hand side equals $t$ and the first sum on the right-hand side is at most $t+2\eta$) we must have
\[(1-a)q\frac{|y_k-y_{k-1}|}{\xi(y_k)} + q\sum_{i=1}^{k-1} \frac{|y_i-y_{i-1}|}{\xi(y_i)} \leq 2\eta.\]
By the triangle inequality, we get
\[|\tilde y-z|=\Big|(1-a)(y_k-y_{k-1}) + \sum_{i=1}^{k-1} (y_i-y_{i-1})\Big| \leq \frac{2\eta}{q}\sup_{y\in B_\nu(0,R)}\xi(y),\]
and by~\ref{asmp:pot_large}$_{R_0}$ and the fact that $\eta \leq \eps/(2R^\gamma)$, we have $|\tilde y-z|\leq \eps$. Since $\tilde y\in s_\nu(t)$ this completes the proof.
\end{proof}

We can now apply Proposition \ref{prop:del_app_general} together with Lemma \ref{le:le61} to prove our main result for this section.

\begin{prop}\label{prop:del_apr_number}
Suppose that $\nu$ satisfies~\ref{asmp:pot_large}$_{R_0}$ and~\ref{asmp:small_pts}$_{r_0}$. For any $\eps > 0$ and $t_0 > 0$, there exists $\delta>0$ (depending only on $\gamma$, $\eps$, $t_0$, $R_0$ and $r_0$) such that
\[m_\nu(z,t) \leq m_\nu^\delta(z,t)\leq m_\nu(z,t)+\eps \quad \mbox{for all } z \in \R^d,\]
and
\[d_H(s_\nu(t), s_\nu^\delta(t))\leq \eps\]
for all $t\in[0,t_0]$ and $z\in\R^d$.
\end{prop}

\begin{proof} We start by showing the statement about the supports, $s_\nu$ and $s_\nu^\delta$. By Lemma \ref{le:le61} we can choose $\eta>0$ such that
\[s_\nu(t+\eta) \subseteq \bigcup_{y \in s_\nu(t)} B(y,\eps) \quad \forall t\leq t_0.\]
Then by Proposition \ref{prop:del_app_general} we can choose $\delta>0$ such that
\[h_\nu^\delta(z) \leq h_\nu(z) \leq h_\nu^\delta(z)+\eta \quad \forall z\in\R^d.\]
We get
\[z\in s_\nu(t) \quad\Rightarrow\quad h_\nu(z)\leq t \quad\Rightarrow\quad h_\nu^\delta(z)\leq t \quad\Rightarrow\quad z\in s_\nu^\delta(t),\]
and
\[z\in s_\nu^\delta(t) \quad\Rightarrow\quad h_\nu^\delta(z)\leq t \quad\Rightarrow\quad h_\nu(z)\leq t+\eta \quad\Rightarrow\quad z\in s_\nu(t+\eta),\]
so
\[s_\nu(t) \subset s_\nu^\delta(t) \subset \bigcup_{y\in s_\nu(t)} B(y,\eps).\]
This implies that $d_H(s_\nu(t),s_\nu^\delta(t))\leq \eps$ as required.

We now turn our attention to the numbers of particles, $m_\nu$ and $m_\nu^\delta$. By Lemma~\ref{le:large_hitting} we can choose $R>R_0$ such that $h_\nu^\delta(z)>t_0$ for all $z\not\in B(0,R)$ and all $\delta\in(0,1]$. Then by Proposition \ref{prop:del_app_general} we can choose $\delta\in(0,1]$ such that $h_\nu^\delta(z)\leq h_\nu(z)\leq h_\nu^\delta(z)+\eps/(qR^\gamma)$ for all $z\in\R^d$. Then, straight from the definitions, we have
\[m_\nu(z,t)\leq m_\nu^\delta(z,t) \leq m_\nu(z,t) + \sup_{y\in B_\nu(0,R)}\xi(y)\frac{\eps}{qR^\gamma}\]
for all $z\in\R^d$ and $t\leq t_0$. By~\ref{asmp:pot_large}$_{R_0}$ the right-hand side is at most $m_\nu(z,t)+\eps$.

Finally, since $h_\nu^\delta(z) \leq h_\nu(z)$ for all $z \in \R^d$ and $h_\nu^\delta$ is increasing as $\delta \downarrow 0$, the 
event $\{ h_\nu^\delta(z) \leq h(z) \leq h_\nu^\delta(z) + \eps \}$, and therefore the events $\{m_\nu(z,t) \leq m_\nu^\delta(z,t)\leq m_\nu(z,t)+\eps\}$ and $\{s_\nu(t) \subset s_\nu^\delta(t) \subset \bigcup_{y\in s_\nu(t)} B(y,\eps)\}$, are increasing as $\delta\downarrow 0$ for any $\eps > 0$. In particular, we can choose the same $\delta$ for both the support and the number of particles.
\end{proof}

\subsection{The $\delta$-approximation works}\label{ssn:delta_works}

Our aim in this section is to show that the $\delta$-approximations converge (in a suitable sense) as $\delta\downarrow0$ to the quantities they are supposed to approximate. In particular we will prove Proposition \ref{prop:appr_works}. We first show that conditions \ref{asmp:pot_large} and \ref{asmp:small_pts} hold with high probability for both $\Pi_T$ and $\Pi$.

\begin{lem}\label{le:conditions_satisfied}
As $R_0\to\infty$,
\[\P\big( \Pi \hbox{ satisfies \ref{asmp:pot_large}$_{R_0}$}) \to 1 \quad\mbox{ and }\quad \liminf_{T\to\infty}\P\Big( \Pi_T \hbox{ satisfies \ref{asmp:pot_large}$_{R_0}$}\Big) \to 1,\]
and as $r_0\to0$,
\[\P\big( \Pi \hbox{ satisfies \ref{asmp:small_pts}$_{r_0}$}) \to 1 \quad\mbox{ and }\quad \liminf_{T\to\infty}\P\Big( \Pi_T \hbox{ satisfies \ref{asmp:small_pts}$_{r_0}$}\Big) \to 1.\]
\end{lem}

\begin{proof} 
Define the event $A_k(\nu) = \{ \max_{z \in A_\nu(0,2^{k}) } \xi(z) \leq q  2^{(k-1) \gamma}\}$. 
By~\cite[Lemma 2.7(ii)]{OR14}, there exists a constant $C$ such that for any $T>e$ and any $k\geq 0$,
\[ \p (A_k(\Pi_T)^c) 
\leq C 2^{dk}  (q2^{\gamma (k-1) } )^{-\alpha} =  C 2^{\alpha\gamma} q^{-\alpha}  2^{(d-\gamma\alpha)k}. \]
Similarly, by direct calculation, there exists a constant $C$ such that for any $R\geq 1$,
\[ \p (A_k(\Pi)^c) 
\leq 1-e^{-C  2^{\alpha\gamma} q^{-\alpha} 2^{(d-\gamma\alpha) k }} \leq C 2^{\alpha\gamma} q^{-\alpha} 2^{(d-\gamma\alpha)k}.\]
Note that $d-\gamma\alpha<0$, so that in  both cases the probabilities are summable over $k$.
In particular, we can choose $K$ large enough so that the event $\cap_{k \geq K } A_k(\nu)$ holds 
with probability arbitrarily close to $1$ (for $\nu = \Pi$ or for $\nu = \Pi_T$ and uniformly in $T > e$).

Now on the event $\cap_{k \geq K } A_k(\nu)$, we can take any $R \geq 2^{K}$ and choose $k$ such that $2^k \leq R \leq 2^{k+1}$. Then, we have
that
\[ \sup_{z \in B(0,R)} \xi(z) \leq \sup_{z \in B(0,2^{k+1})} \xi(z) \leq q  2^{k\gamma} \leq q R^\gamma ,\]
so that the first statement follows.

To show~\ref{asmp:small_pts}$_{r_0}$, we define $\tilde A_k(\nu) = \{ \exists z \in B_\nu(0,2^{-k}) \, : \, \xi(z) \geq 2^{-\gamma (k-1)}  \}$. 
For $\nu = \Pi_T$, we have from~\cite[Lemma~2.7(i)]{OR14} that there exists $c > 0$ such that for $T > e$,
\[\p ( \tilde A_k(\Pi_T)^c )  = \p \Big( \max_{y  \in \supp \Pi_T^\ssup{1} \cap B(0,2^{-k})}\xi(y) \leq  2^{-\gamma (k-1)} \Big) \leq e^{-c 2^{-\alpha \gamma} 2^{k(\alpha \gamma - d)}  }.\]
Similarly, by direct calculation, there exists a constant $c>0$ such that
\[\p ( \tilde A_k(\Pi)^c ) = \p \Big( \max_{y  \in \supp \Pi^\ssup{1} \cap B(0,2^{-k})}\xi(y) \leq 2^{-\gamma (k-1)} \Big) \leq e^{-c 2^{-\alpha \gamma} 2^{k(\alpha \gamma - d)}  }.\]
Note that $\alpha \gamma - d>0$, so for any $\eps>0$ we can choose $K$ such that for all large $T$,
\[ \P\Big(\bigcup_{k\geq K} \tilde A_k(\Pi_T)^c \Big) \leq \eps \quad\mbox{ and }\quad \P\Big(\bigcup_{k\geq K} \tilde A_k(\Pi)^c \Big) \leq \eps.\]
The result follows.
\end{proof}

We also note the following easy lemma.

\begin{lem}\label{le:lilypad_nontrival}
Almost surely, $h_\Pi(z) \in (0,\infty)$ and $h_{\Pi_T}(z)\in(0,\infty)$ for any $z \not= 0$.
\end{lem}

\begin{proof}
The statement follows by combining Corollary~\ref{cor:nontriv} with Lemma~\ref{le:conditions_satisfied}.
\end{proof}

The next corollary is the key tool in proving Proposition \ref{prop:appr_works}.

\begin{cor}\label{cor:approx_works} For any $\eps > 0$, $T>e$ and $t_0 > 0$,
\begin{eqnarray}
&\lim_{\delta \downarrow 0 }  \p \Big( \sup_{z \in \R^d} | h_\Pi(z)  - h_\Pi^\delta(z) | \geq \eps \Big)  =  0, \\
&\lim_{\delta \downarrow 0 }  \p \Big( \sup_{t \leq t_0} \sup_{z \in \R^d} | m_\Pi(z,t)  - m_\Pi^\delta(z,t) | \geq \eps \Big)  =  0,\\
&\lim_{\delta \downarrow 0 }  \p \Big( \sup_{t \leq t_0} d_{H} (s_\Pi(t) , s_\Pi^\delta(t) ) \geq \eps \Big)  =  0,
 \end{eqnarray}
 and similarly
\begin{eqnarray}
&\lim_{\delta \downarrow 0 } \limsup_{T\to\infty} \p \Big( \sup_{z \in \R^d} | h_{\Pi_T}(z)  - h_{\Pi_T}^\delta(z) | \geq \eps \Big) = 0 , \\
&\lim_{\delta \downarrow 0 } \limsup_{T\to\infty} \p \Big( \sup_{t \leq t_0} \sup_{z \in \R^d} | m_{\Pi_T}(z,t)  - m_{\Pi_T}^\delta(z,t) | \geq \eps \Big)  =  0,\\
&\lim_{\delta \downarrow 0 } \limsup_{T\to\infty} \p \Big( \sup_{t \leq t_0} d_{H} (s_{\Pi_T}(t) , s_{\Pi_T}^\delta(t) ) \geq \eps \Big)  =  0.
 \end{eqnarray}
\end{cor}

\begin{proof}
First, since $h_\Pi^\delta(z)\leq h_\Pi(z)$ for all $z\in\R^d$ and $\delta>0$, and $h_\Pi^\delta(z)$ is increasing as $\delta\downarrow0$, the events $\{h_\Pi^\delta(z)\leq h(z)\leq h_\Pi^\delta(z)+\eps\}$ are increasing as $\delta\downarrow0$. By Lemma \ref{le:conditions_satisfied} and Proposition \ref{prop:del_app_general}, we know that for any $\eps>0$,
\[\lim_{\delta\downarrow 0} \P( h_\Pi^\delta(z)\leq h_\Pi(z)\leq h_\Pi^\delta(z)+\eps \quad \forall z\in\R^d) = 1\]
and
\[\lim_{\delta\downarrow 0} \liminf_{T\to\infty}\P( h_{\Pi_T}^\delta(z)\leq h_{\Pi_T}(z)\leq h_{\Pi_T}^\delta(z)+\eps \quad \forall z\in\R^d) = 1;\]
the first and fourth statements follow. The proofs of the statements for $m$ and $s$ are almost identical, using Proposition \ref{prop:del_apr_number} in place of Proposition \ref{prop:del_app_general}.
\end{proof}

From Corollary~\ref{cor:approx_works}, we can easily deduce our main technical result Proposition~\ref{prop:appr_works}.

\begin{proof}[Proof of Proposition~\ref{prop:appr_works}]
We consider first the case of the hitting times. Recall that we defined, for any $f , g\in \mathcal{C}^d := C(\R^d,[0,\infty))$,
 \[ d_{U}(f,g) = \sum_{k \geq 1} 2^{-k} \Big( \sup_{x \in [-k,k]^d} \big\{|f(x) - g(x)|\big\} \wedge 1 \Big). \]
For any $\eps > 0$, we choose $N$ such that $2^{-N} \leq \eps/2$. Then we have
\[ \begin{aligned} \p \big( d_{U} (h^\delta_T, h_T) \geq \eps \big) 
& \leq \p \Big( \sum_{k=1}^N  \sup_{z \in [-k,k]^d} |h_T^\delta(z) - h_T(z)|  \geq \eps/2 \Big) \\
& \leq \sum_{k=1}^N \p\Big( \sup_{z \in [-k,k]^d} |h_T^\delta(z) - h_T(z)|  \geq \eps/(2N) \Big)
\end{aligned} 
\]
Letting first $T \ra \infty$ and then $\delta \downarrow 0$, we obtain by Corollary~\ref{cor:approx_works} that
\[ \lim_{\delta \downarrow 0} \limsup_{T \ra\infty}  \p \big( d_{U} (h^\delta_T, h_T) \geq \eps \big)  = 0 . \]
The argument for the numbers of particles and the support of the discrete lilypad model as well as the analogous statements
for the Poisson lilypad model also follow from Corollary~\ref{cor:approx_works} in exactly the same way.
If we combine these statements, we obtain Proposition~\ref{prop:appr_works}.
\end{proof}

\subsection{Proof of Theorem~\ref{thm:weak_conv}}\label{ssn:proof_scaling}

We would like to apply the continuous mapping theorem to deduce the weak convergence of the $\delta$-truncated lilypad models. To facilitate this application, we introduce some slightly different $\delta$-approximations: define, for $z\in\R^d$ and $\delta>0$,
\[ \tilde h^\delta_\nu(z)  = \inf \Big\{ \sum_{j=1}^n q \frac{|y_{j-1} - y_j|}{\xi(y_j)} + q\frac{|y_n|}{\delta}  \, : \,n\in\N_0, y_0 = z \mbox{ and }y_1, \ldots, y_{n} \in \cB_\nu(0,1/\delta) \Big\}.\]
Note that the only difference from our previous definition $h^\delta_\nu$ is that the points $y_1\ldots,y_n$ must now be within the closed ball $\cB(0,1/\delta)$. We also define
\[ \tilde m_\nu^\delta(z,t) = \sup_{y \in \cB_\nu(0,1/\delta)} \Big\{ \xi(y) ( t - \tilde h_\nu^\delta(y) ) - q |y-z| \Big\} \vee 0 , \quad z \in \R^d, t \geq 0, \]
and
\[ \tilde s_\nu^\delta(t) = \{ z \in \R^d \,  : \, \tilde h_\nu^\delta(z) \leq t \} , \quad  t \geq 0 . \]

We recall that $h_T^\delta$ is shorthand for $h_{\Pi_T}^\delta$, $h^\delta$ for $h_\Pi^\delta$, and so on; and we similarly write $\tilde h_T^\delta$ for $\tilde h_{\Pi_T}^\delta$, $\tilde h^\delta$ for $\tilde h_\Pi^\delta$ and so on.

The benefit of introducing these new quantities is that applying the continuous mapping theorem to them is straightforward.

\begin{prop}\label{prop:d_apr_conv} For any $\delta > 0$,  as $T \ra \infty$
\[ (\tilde h_T^\delta, \tilde m_T^\delta, \tilde s_T^\delta) \Rightarrow (\tilde h^\delta, \tilde m^\delta, \tilde s^\delta) . \]
\end{prop}

\begin{proof} As discussed in Section~\ref{ssn:point_pr},
we know that 
\[ \Pi_T \Rightarrow \Pi . \]
 By the continuous mapping theorem, \cite[Theorem 5.1]{Bil68}, 
 we only have to show that
each of the maps
\[ \nu \mapsto \tilde h_\nu^\delta, \quad \nu \mapsto \tilde m_\nu^\delta, \quad \nu \mapsto \tilde s^\delta_\nu, \]
are continuous as functions from $M_p(E)$ (equipped with the vague topology) into the target spaces
equipped with the topologies described before Theorem~\ref{thm:weak_conv}.

We note that the definitions of $\tilde h^\delta_\nu$, $\tilde m^\delta_\nu$, and $\tilde s^\delta_\nu$ only depend on the point process through 
the values in $\cB(0,1/\delta) \times [\delta , \infty)$, which is a compact set in $E$. 
The same is true for $\tilde m^\delta_\nu$ and $\tilde s^\delta_\nu$.
Therefore, we can use Proposition 3.31 in~\cite{Res08}:
given that $\nu_n$ converges vaguely to $\nu$,  we can label 
atoms of $\nu_n$ and $\nu$ restricted to any compact set such that the finitely many atoms converge pointwise. 
This implies in particular that $\tilde h_{\nu_n}^\delta \ra \tilde h_\nu^\delta$, $\tilde m_{\nu_n}^\delta \ra \tilde m_\nu^\delta$, and $\tilde s_{\nu_n}^\delta \ra \tilde s_\nu^\delta$.
\end{proof}

Write
\[ \Triple_T  = (H_T, M_T, S_T),\quad \triple_T = (h_T, m_T, s_T), \quad \triple_T^\delta = (h_T^\delta, m_T^\delta, s_T^\delta), \quad \ttriple_T^\delta = (\tilde h_T^\delta, \tilde m_T^\delta, \tilde s_T^\delta), \]
\[\ttriple^\delta = (\tilde h^\delta,\tilde m^\delta, \tilde s^\delta), \quad \triple^\delta  = (h^\delta ,m^\delta, s^\delta) , \quad \triple = (h,m,s) .\]
We now need to check that $\ttriple^\delta_T$ is close to $\triple^\delta_T$, and $\ttriple^\delta$ is close to $\triple^\delta$.

\begin{lem}\label{le:tildegood}
For any $\eps>0$,
\[\lim_{\delta\downarrow0} \limsup_{T\to\infty}\P( d^\ssup{\times 3}(\ttriple^\delta_T, \triple^\delta_T)>\eps) = 0 \quad\mbox{ and }\quad \lim_{\delta\downarrow0} \P( d^\ssup{\times 3}(\ttriple^\delta, \triple^\delta)>\eps) = 0.\]
\end{lem}

\begin{proof}
Fix $\eta>0$; by Lemma \ref{le:conditions_satisfied} we may choose $R_0, r_0>0$ such that both $\Pi_T$ (for any large $T$) and $\Pi$ satisfy \ref{asmp:pot_large}$_{R_0}$ and \ref{asmp:small_pts}$_{r_0}$ with probability at least $1-\eta$.

By Lemmas \ref{le:small_times} and \ref{le:large_hitting}, for any point measure $\nu$ satisfying \ref{asmp:pot_large}$_{R_0}$ and \ref{asmp:small_pts}$_{r_0}$, and any $R>0$ and $t_0>0$, there exists $\delta_0>0$ such that for all $\delta\in(0,\delta_0)$,
\[\inf_{y\not\in \cB(0,1/\delta)} h^\delta_\nu (y) > \max\Big\{ \sup_{z\in B(0,R)} h^\delta_\nu(z), t_0\Big\}.\]
Then for all $\delta\in(0,\delta_0)$, $z\in B(0,R)$ and $t\leq t_0$, we have
\[\tilde h^\delta_\nu(z) = h^\delta_\nu(z), \quad \tilde m^\delta_\nu(z,t) = m^\delta_\nu(z,t) \quad \mbox{and}\quad \tilde s^\delta_\nu(t) = s^\delta_\nu(t).\]
From the definition of $d^\ssup{\times 3}$ (choosing $R$ and $t_0$ large enough that the distance is guaranteed to be small) we get that for all large $T$,
\[\p\Big( d^\ssup{\times 3}(\ttriple^\delta_T, \triple^\delta_T)>\eps\Big) \leq\eta \quad\mbox{ and }\quad \p\Big( d^\ssup{\times 3}(\ttriple^\delta, \triple^\delta)>\eps\Big)\leq \eta\]
for all $\delta\in(0,\delta_0)$. Since $\eta>0$ was arbitrary, this completes the proof.
\end{proof}

We can now combine the various parts of this section to deduce the main scaling limit, Theorem~\ref{thm:weak_conv}.

\begin{proof}[Proof of Theorem~\ref{thm:weak_conv}]
By the portmanteau theorem it suffices to show that for any
bounded and Lipschitz-continuous function $f : \cC^\ssup{\times 3} \ra \R$, we
have that
\begin{equation}\label{eq:wk_conv_lip} \E [ f(H_T, M_T, S_T) ] \ra \E [ f ( h, m , s) ] \quad \mbox{as } T \ra \infty. \end{equation}  
Suppose that $f : \cC^3 \ra \R$ is bounded by $\|f\|$ and Lipschitz continuous with Lipschitz constant $L$, and let $\eps > 0$.
We have that 
\[\begin{aligned}  \big| \E [ f(\Triple_T) ] - \E [f (\triple) ] \big|
& \leq  \E \big[ | f(\Triple_T) - f(\triple_T) | \big] 
 + 
\E \big[ | f(\triple_T) - f (\triple_T^\delta) |\big] + \E\big[|f(\triple_T^\delta)-f(\ttriple_T^\delta)|\big]
\\
&\quad + \big| \E [ f(\ttriple_T^\delta )] - \E [ f(\ttriple^\delta) ] \big| 
+ \E\big[ | f(\ttriple^\delta ) - f(\triple^\delta) | \big] + \E\big[|f(\triple^\delta)-f(\triple)|\big] \\
& \leq 5 L \eps + 2\| f \| \, \p ( d^\ssup{\times 3}( \Triple_T, \triple_T) > \eps ) + 2\| f \| \, \p ( d^\ssup{\times 3}( \triple_T, \triple_T^\delta) > \eps ) \\
& \qquad + 2\|f\|\,\p(d^\ssup{\times 3}(\triple^\delta_T,\ttriple^\delta_T)>\eps) +
\big| \E [ f(\ttriple_T^\delta )] - \E [ f(\ttriple^\delta) ] \big| \\
&\qquad + 2\|f\|\,\p(d^\ssup{\times 3}(\ttriple^\delta,\triple^\delta)>\eps) + 2\| f \|\, \p ( d^\ssup{\times 3}(\triple^\delta, \triple) > \eps ). \\
\end{aligned} \]
We now take a $\limsup$ as $T\to\infty$: by Theorem \ref{thm:discrete_lilypad},
\[\p ( d^\ssup{\times 3}( \Triple_T, \triple_T) > \eps ) \to 0;\]
and by Proposition \ref{prop:d_apr_conv},
\[\big| \E [ f(\ttriple_T^\delta )] - \E [ f(\ttriple^\delta) ] \big|\to 0.\]
Thus
\begin{multline*}
\limsup_{T\to\infty} \big| \E [ f(\Triple_T) ] - \E [f (\triple) ] \big|\\
\leq 5 L \eps + 2\| f \| \, \limsup_{T\to\infty}\p ( d^\ssup{\times 3}( \triple_T, \triple_T^\delta) > \eps ) + 2\|f\|\,\limsup_{T\to\infty}\p(d^\ssup{\times 3}(\triple^\delta_T,\ttriple^\delta_T)>\eps) \\
\qquad + 2\|f\|\,\p(d^\ssup{\times 3}(\ttriple^\delta,\triple^\delta)>\eps) + 2\| f \|\, \p ( d^\ssup{\times 3}(\triple^\delta, \triple) > \eps ).
\end{multline*}
Finally, by Proposition \ref{prop:appr_works} and Lemma \ref{le:tildegood}, taking a limit as $\delta\downarrow0$ on the right-hand side, we get
\[\limsup_{T\to\infty} \big| \E [ f(\Triple_T) ] - \E [f (\triple) ] \big| \leq 5 L \eps,\]
and since $\eps>0$ was arbitrary the proof is complete.
\end{proof}

\section{Proof of the ageing result}\label{sec:proof_ageing}

In this section we prove Theorem~\ref{thm:ageing}. 

Before we start with the main proof, we need to collect several auxiliary lemmas, 
where we show that the lilypad models are rather `discrete':  once
two maximizing points are close, they are in fact the same.

\begin{lemma}\label{le:compact_max} For any $t > 0$
\[ \lim_{n \ra \infty}  \limsup_{T \ra \infty } \p ( S_T(\cdot,t) \not\subseteq B(0,n) )   =  0 . \]
\end{lemma}

\begin{proof}
Follows for $s_T$ instead of $S_T$ by combining Lemma~\ref{le:large_hitting} with Lemma~\ref{le:conditions_satisfied}
and thus for $S_T$ by Theorem~\ref{thm:discrete_lilypad}.
\end{proof}

\begin{lemma}\label{le:Poisson_simple}
We have:
\begin{itemize}
	\item[(i)] For any $t > 0$, $\lim_{n \ra \infty} \p (\supp m(\cdot, t) \not\subseteq B(0,n)) = 0 )$. 
	\item[(ii)] For any $t > 0$, $ \lim_{\eps \downarrow 0 } \p (\xi(w(t)) \leq \eps ) = 0$.
	\item[(iii)] For any $n \in \N, \eps > 0$,
	\[ \lim_{\delta \downarrow 0}  \p ( \exists z_1 \neq z_2 \in B_\Pi(0,n) \, : \, |z_1 - z_2|< \delta \mbox{ and } \xi(z_1) \geq \eps , \xi(z_2) \geq \eps )  = 0 . \]
\end{itemize}
\end{lemma}

\begin{proof}
(i) Follows by combining Lemma~\ref{le:large_hitting} with Lemma~\ref{le:conditions_satisfied}.

(ii) By monotone convergence
\[  \lim_{\eps \downarrow 0 } \p (\xi(w(t)) \leq \eps ) = \p (\xi(w(t)) = 0 )  = \p ( m(x,t) = 0 \mbox{ for all } x) . \]
But by Lemma~\ref{le:lilypad_nontrival} we know that the lilypad model is almost surely non-trivial, so the latter probability is $0$.

(iii) By the standard Palm calculus for Poisson processes we know that, conditionally on $\Pi(\{ (x,y) \}) = 1$, the process
$\Pi - \delta_{(x,y)}$ is again a Poisson process with intensity $\pi$, see e.g.~\cite[Theorem 3.1]{Baddeley}. 
Therefore, we can write
\[ \begin{aligned} 	\p ( \exists z_1 & \neq z_2 \in B(0,n) \, : \, z_2 \in B(z_1, \delta)  \mbox{ and } \xi(z_1) \geq \eps , \xi(z_2) \geq \eps ) \\
& = \int_{B(0,n) \times [\eps, \infty)} \p( \exists z_2 \in B(z_1, \delta)\setminus \{z_1\}  \, : \, \xi(z) \geq \eps ) 
\pi( d (x_1,y_2) )  \\
& = \int_{B(0,n) \times [\eps, \infty)} \p ( \Pi( B(z_1, \delta) ) \times [\eps, \infty) \neq 0 ) ) 
 \pi( d (x_1,y_2) )  \\
\end{aligned} \] 
However, we know that
\[ \p ( \Pi( B(z_1, \delta) ) \times [\eps, \infty) \neq 0 )
= 1 - e^{ - \pi( B(z_1, \delta) ) \times [\eps, \infty) } \ra 0 , \]
as $\delta \downarrow 0$. The claim follows by dominated convergence, since $\pi( B(0,n) \times [\eps, \infty)) < \infty$.
\end{proof}

\begin{lem}\label{le:close_not_equal} For any $0 \leq s < t$,
\[\lim_{\delta \downarrow 0}  \limsup_{T \ra \infty}  \p ( | W_T(t) - W_T(s)| < \delta ; W_T(t) \neq W_T(s) ) = 0 \]
and
\[\lim_{\delta \downarrow 0} \p ( | w(t) - w(s)| < \delta ; w(t) \neq w(s) ) = 0. \]
\end{lem}

\begin{proof} We begin with the first statement. From Theorem~1.1 in~\cite{onepoint}, we know  that for any $t$, 
with probability tending to $1$, the branching random walk is localised in 
the maximizer $w_T(t)$ of $m_T(\cdot, t)$. Therefore it suffices to show 
the corresponding statement for $w_T(t)$.

Note that for any $t > 0$
\[ \lim_{\eps \downarrow 0} \limsup_{T \ra \infty} \p ( \xi(w_T(t)) \geq \eps ) \leq \lim_{\eps\downarrow 0}\p ( m(w(t), t) \leq \eps )
= 0, \]
since the limiting model $m(\cdot, t)$ is almost surely non-trivial by Lemma~\ref{le:lilypad_nontrival}.
Also, by  Lemma~\ref{le:compact_max}, we have for any $t$ that
\[ \lim_{n \ra \infty} \limsup_{T \ra \infty} \p ( |w_T(t)| \geq n) = 0 . \]
Now, for fixed $s$ and $t$, under the assumptions that $\xi(w_T(t)) \wedge \xi(w_T(s)) > \eps$
and $|w_T(t)| \vee |w_T(s)| < n$, the event
$\{ | w_T(t) - w_T(s)| < \delta ;\, w_T(t) \neq w_T(s)\}$
implies that there exist $w \neq w' \in  L_T(0,n)$ with $|w - w'| \leq \delta$ such that
$\xi_T(w), \xi_T(w') \geq \eps$. Thus, by the above, we are done if we can show that for any 
$n \in \N, \eps > 0$,
\[ \lim_{\delta \downarrow 0} \limsup_{T\to\infty}\p ( \exists w \neq w' \in L_T(0,n) \, : \, 
|w - w'| \leq \delta \, : \, \xi_T(w), \xi_T(w') \geq \eps ) = 0 . \]
However, this follows from an explicit calculation: for some constant $C$,
\begin{multline*}\p ( \exists w \neq  w' \in L_T(0,n) \, : \, 
|w - w'| \leq \delta \, : \, \xi_T(w), \xi_T(w') \geq \eps ) \\
  \leq C r(T)^{2d} a(T)^{- 2\alpha} n^d \delta^d \eps^{-2\alpha} 
=  C  n^d \delta^d \eps^{-2\alpha} , 
\end{multline*}
and letting $T\to\infty$ and then $\delta \downarrow 0$ completes the proof of the first statement. The second is almost identical, using Lemma \ref{le:Poisson_simple}.
\end{proof}

We now check that the maximizer for the Poisson lilypad model behaves sensibly. For $x\in\R^d$ and $\delta>0$, let $\partial B(x,\delta) = \{z : |z-x|=\delta\}$, the boundary of the ball of radius $\delta$ about $x$.

\begin{lemma}\label{le:Poisson_max_unique}
The following are true:
\begin{itemize}
\item[(i)]
For any $t \geq 0$, almost surely, there is a single maximizer in the Poisson model 
$m(\cdot, t)$. 
\item[(ii)] For any fixed $x\in\R^d$, $\delta > 0$ and $t> 0$, $\Px\big(w(t) \in \partial B(x,\delta)\big) =0$.
\end{itemize}
\end{lemma}

\begin{proof}
(i) The basic idea is the following: if both $w$ and $w'$ are maximizers, we have $m(w,t)=m(w',t)$, which means $\xi(w) = \xi(w')(t-h(w'))/(t-h(w))$. Suppose without loss of generality that $h(w)\geq h(w')$. Then from the definition of $h$, if $w\neq w'$, the values of $\xi(w')$, $h(w')$ and $h(w)$ are independent of $\xi(w)$. So the probability that $\xi(w)$ takes on the exact value $\xi(w')(t-h(w'))/(t-h(w))$ is zero.

However, since our point process $\Pi$ has infinitely many atoms, we need to be careful.

Fix for a moment $z\in\R^d$, $\delta>0$ and $\eps>0$, and let $\hat \Pi$ be the point process obtained by taking $\Pi$ and removing all of the points in $B(z,\delta)\times(\eps,\infty)$ and $\tilde \Pi$ be the point process consisting of only those points of $\Pi$ in $B(z,\delta)\times(\eps,\infty)$. Clearly $\hat \Pi$ and $\tilde \Pi$ are independent.

Note from the definition of $h$ that for any $w\in B(z,\delta)$, if $B_{\tilde \Pi}(z,\delta)=\{w\}$, then $h_\Pi(w) = h_{\hat\Pi}(w)$. Furthermore, for any other point $w'\in\R^d$, if additionally $h_\Pi(w)\geq h_\Pi(w')$ then $h_\Pi(w') = h_{\hat\Pi}(w')$. Therefore
\begin{align*}
&\P\big( \exists w\in\supp\tilde\Pi^\ssup{1},\, w'\in\supp\hat\Pi^\ssup{1} \,:\, B_{\tilde \Pi}(z,\delta)=\{w\},\, h_\Pi(w)\geq h_\Pi(w'), \\
&\hspace{70mm} \xi_\Pi(w)(t-h_\Pi(w)) = \xi_\Pi(w')(t-h_\Pi(w'))\big)\\
&\leq \P\big(\exists w\in\supp\tilde\Pi^\ssup{1},\, w'\in\supp\hat\Pi^\ssup{1} \,:\, B_{\tilde \Pi}(z,\delta)=\{w\}, \\
&\hspace{70mm} \xi_{\tilde\Pi}(w)(t-h_{\hat\Pi}(w)) = \xi_{\hat\Pi}(w')(t-h_{\hat\Pi}(w'))\big)\\
&= 0,
\end{align*}
since $\hat\Pi$ and $\tilde\Pi$ are independent. Returning to our usual notation, this tells us that
\begin{multline*}
\P(\exists w\in B_\Pi(z,\delta),\, w'\in\supp \Pi^\ssup{1} \,:\, h(w)>h(w'),\,\xi(w)>\eps,\\
\xi(y)\leq \eps\,\,\forall y\in B_\Pi(z,\delta)\setminus\{w\},\, m(w,t)=m(w',t)) = 0
\end{multline*}
(where no subscript means we are using the point process $\Pi$).

Now, taking a sum over all $z$ such that $z/\delta\in\Z^d\cap B(0,n)$, we deduce that
\begin{multline*}
\P(\exists w, w'\in B_{\Pi}(0,n) \,:\, h(w)>h(w'),\,\xi(w)>\eps,\\
\xi(y)\leq \eps\,\,\forall y\in B_\Pi(w,2\delta)\setminus\{w\},\, m(w,t)=m(w',t)) = 0.
\end{multline*}
Taking a limit as $\delta\downarrow0$, we get by Lemma \ref{le:Poisson_simple} (iii) that
\[\P(\exists w, w'\in B_{\Pi}(0,n) \,:\, h(w)>h(w'),\,\xi(w)>\eps,\, m(w,t)=m(w',t)) = 0.\]
Now taking $n\to\infty$, by Lemma \ref{le:Poisson_simple} (i), we have
\[\P(\exists w, w'\in \supp \Pi^\ssup{1} \,:\, h(w)>h(w'),\,\xi(w)>\eps,\, m(w,t)=m(w',t)>0) = 0.\]
Finally, taking $\eps\downarrow0$, by Lemma \ref{le:Poisson_simple} (ii), we get
\[\P(\exists w, w'\in \supp \Pi^\ssup{1} \,:\, h(w)>h(w'),\, m(w,t)=m(w',t)=\sup_{x\in\Z^d}m(x,t)) = 0.\]
This completes the proof of (i).

(ii) To show the second statement, we note that by construction the maximizer $w(t)$
is in $\supp \Pi^\ssup{1}$. Also, note that the event $\{ \xi(w(t)) \geq \eps \}$ is an increasing event as $\eps \downarrow 0$. 
Thus, we have by Lemma~\ref{le:Poisson_simple} and monotone convergence
\[ \begin{aligned} \p ( w(t) \in \partial B(x,\delta) ) &
= \lim_{\eps \downarrow 0} \p ( w(t) \in \partial B(x,\delta) , \xi(w(t) )\geq \epsilon )\\
& \leq \limsup_{\eps \downarrow 0 }  \p ( w(t) \in \partial B(x,\delta) , \xi(w(t)) \geq \eps ) \\
& \leq \limsup_{\eps \downarrow 0} \p ( \Pi( \partial B(x,\delta) \times [\eps,\infty)) \geq 1 ) 
= 0, 
\end{aligned} 
\] 
since $\pi( \partial B(x,\delta) \times [\eps,\infty)  ) = 0$.
\end{proof}

\begin{lemma}\label{le:Poisson_not_equal} For any $\theta > 0$,
\[ \lim_{\delta \downarrow 0}\P \big(  |w(1) -  w(1+\theta)| \leq 2\delta,\,\, w(1) \neq w(1+\theta)\big)  = 0.\]
\end{lemma}

\begin{proof} Let $n \in \N$ and $\eps > 0$. Then 
\[\begin{aligned}  \P (  & |w(1)  -  w(1+\theta)| \leq  2\delta, w(1) \neq w(1+\theta))  \\
& \leq  \P (  |w(1)  -  w(1+\theta)| \leq  2\delta, w(1) \neq w(1+\theta), \xi(w(1)) \geq \eps, \xi(w(1+\theta)) \geq \eps) \\
& \hspace{2cm} + \p (\min \{ \xi (w(1)), \xi(w(1+\theta))\}  \leq \eps) \\
& \leq \p ( \exists z_1 \neq z_2 \in B_\Pi(0,n) \, : \, | z_2 - z_1| \leq 2 \delta ,  \xi(z_1) \geq \eps , \xi(z_2) \geq \eps ) \\
& \hspace{2cm} + \p (\min \{ \xi (w(1)), \xi(w(1+\theta))\}  \leq \eps) 
+ \p ( \max \{ |w(1)|, |w(1+\theta)|\} \geq n ) . 
\end{aligned}
\]
Now, letting $\delta \downarrow 0$, we obtain from Lemma~\ref{le:Poisson_simple}(iii) that
\[ \begin{aligned}   \lim_{\delta \downarrow 0} \P (  & |w(1)  -  w(1+\theta)| \leq  2\delta, w(1) \neq w(1+\theta))   \\
& \leq \p (\min \{ \xi (w(1)), \xi(w(1+\theta))\}  \leq \eps) 
+ \p ( \max \{ |w(1), |w(1+\theta)| \geq n ) . 
\end{aligned}
\] 
Finally, letting $\eps \downarrow 0$ and $n \ra \infty$, we obtain the statement 
from Lemma~\ref{le:Poisson_simple} (i), (ii).
\end{proof}

We are now finally ready to prove the ageing result, Theorem~\ref{thm:ageing}.

\begin{proof}[Proof of Theorem~\ref{thm:ageing}]
We start with a lower bound.
For any $\theta > 0, \delta > 0$, define the open  set
\[\begin{aligned} \cO_\theta^\delta  := \Big\{ f \in \cC^{d+1}_0 \, : \, 
 \exists y \in \R^d \, \hbox{ with } 
 &  \max_{z \in \R^d \setminus B(y,\delta) } f(z,1) < f(y,1),  \\
 & \max_{z \in \R^d \setminus B(y,\delta) } f(z,1+\theta) < f( y,1+\theta) 
 \Big\}  . 
 \end{aligned}
\]
From the weak convergence $M_T \Rightarrow m$, we know that
\begin{equation}\label{eq:Oli}
\liminf_{T \ra \infty} \p ( M_T \in \cO_\theta^\delta ) \geq \p (m \in \cO_\theta^\delta) .
\end{equation}

Note that if $w(1)=w(1+\theta)$, then $m\in\cO_\theta^\delta$ for any $\delta>0$; so
\begin{align*}
\P(w(1) = w(1+\theta)) &= \P(w(1) = w(1+\theta), m\in\cO_\theta^\delta)\\
&= \P(m\in\cO_\theta^\delta)-\P(m\in\cO_\theta^\delta, w(1)\neq w(1+\theta)).
\end{align*}
Note that on the event $\{m\in\cO_\theta^\delta\}$, if there are two different maximizers at times $1$ and $1+\theta$ then they must be within distance $\delta$. Thus by Lemma \ref{le:close_not_equal}, $\lim_{\delta\downarrow0} \P(m\in\cO_\theta^\delta, w(1)\neq w(1+\theta)) = 0$, and therefore
\begin{equation}\label{eq:wo1}
\P(w(1) = w(1+\theta)) = \lim_{\delta\downarrow0} \P(m\in\cO_\theta^\delta).
\end{equation}

Similarly, for any $\delta>0$,
\begin{align*}
\P(W_T(1) = W_T(1+\theta)) &= \P(W_T(1) = W_T(1+\theta), M_T\in\cO_\theta^\delta)\\
&= \P(M_T\in\cO_\theta^\delta)-\P(M_T\in\cO_\theta^\delta, W_T(1)\neq W_T(1+\theta)),
\end{align*}
and by Lemma \ref{le:close_not_equal},
\[\lim_{\delta\downarrow 0} \limsup_{T\to\infty}\P(M_T\in\cO_\theta^\delta, W_T(1)\neq W_T(1+\theta)) = 0\]
since on the event $\{M_T\in\cO_\theta^\delta\}$, if there are two different maximizers at times $1$ and $1+\theta$ then they must be within distance $\delta$. Therefore
\[\liminf_{T\to\infty} \P(W_T(1) = W_T(1+\theta)) = \lim_{\delta\downarrow0}\liminf_{T\to\infty} \P(M_T\in \cO_\theta^\delta).\]
Combining this with \eqref{eq:Oli} and \eqref{eq:wo1}, we get
\[\lim_{\delta\downarrow 0} \liminf_{T\to\infty} \P(W_T(1) = W_T(1+\theta)) \geq \P(w(1) = w(1+\theta)),\]
which is the required lower bound.

We now continue with an upper bound. Recall that $\cB(z,r)$ is the closed ball of radius $r$ about $z$. For $z \in \R^d$, $\delta > 0$ and $\theta > 0$, we consider the set
\[\begin{aligned} C_\theta(z,\delta) := \Big\{ f \in \cC_0^{d+1} &  \, : \, 
   \max_{x\in \cB(z,\delta)} f(x,1) = \max_{x \in \R} f(x,1) , \\
   & \max_{x\in \cB(z,\delta)} f(x,1+\theta) = \max_{x \in \R} f(x,1+\theta) \Big\} .
  \end{aligned}
\]
This set is closed, 
so since $M_T \Rightarrow m$ we know that
\begin{equation}\label{eq:0206-1} \limsup_{T \ra \infty} \p ( M_T \in C_\theta(z,\delta) ) \leq \p (m \in C_\theta(z,\delta) ) . \end{equation}

Now let $n \in \N, \delta > 0$ and take $\Gamma_n^\delta$ to be a collection of points such that
$ \cB(0,n\delta) = \bigcup_{z \in \Gamma_n^\delta}  \cB(z,\delta)$, but the collection $\{B(z,\delta) : z \in \Gamma_n^\delta\}$ is disjoint (recall that we are working with $L^1$-balls so that this is possible). Then
\[\p ( W_T(1) = W_T(1+\theta) ) \leq \sum_{z \in \Gamma_n^\delta} \p ( M_T \in C_\theta(z,\delta)) + \p (W_T(1) \notin  \cB(0,n\delta) ),\]
and combining with \eqref{eq:0206-1} and Lemma \ref{le:compact_max} we get that for any $\delta>0$,
\begin{equation}\label{eq:wc}
\limsup_{T\to\infty} \p ( W_T(1) = W_T(1+\theta) ) \leq \limsup_{n\to\infty}\sum_{z \in \Gamma_n^\delta} \P( m\in C_\theta(z,\delta)).
\end{equation}

On the other hand, since by Lemma \ref{le:Poisson_max_unique} the maximizers for the Poisson lilypad model at times $1$ and $1+\theta$ are almost surely unique and not located on the boundary of any of the balls $B(z,\delta)$ for $z\in \Gamma_n^\delta$, we have
\begin{align*}
\sum_{z\in\Gamma_n^\delta} \P(m\in C_\theta(z,\delta)) &\leq \sum_{z\in\Gamma_n^\delta} \P(|w(1)-z|\leq \delta, |w(1+\theta)-z|\leq \delta)\\
&\leq \P(\exists z\in B(0,n\delta) : |w(1)-z|\leq \delta, |w(1+\theta)-z|\leq \delta).
\end{align*}
But, for any $n$,
\begin{multline*}
\P(\exists z\in B(0,n\delta) : |w(1)-z|\leq \delta, |w(1+\theta)-z|\leq \delta)\\
\leq \P(w(1) = w(1+\theta)) + \P(w(1)\neq w(1+\theta), |w(1)-w(1+\theta)|\leq 2\delta),
\end{multline*}
and by Lemma \ref{le:Poisson_not_equal}, the limit of the latter probability as $\delta\downarrow0$ is zero. Thus
\[\lim_{\delta\downarrow0}\limsup_{n\to\infty}\sum_{z\in\Gamma_n^\delta} \P(m\in C_\theta(z,\delta)) \leq \P(w(1) = w(1+\theta)).\]
Combining this with \eqref{eq:0206-1} and \eqref{eq:wc}, we obtain
\[\limsup_{T\to\infty} \p ( W_T(1) = W_T(1+\theta) ) \leq \P(w(1) =w(1+\theta )), \]
which is the required upper bound and completes the proof.
\end{proof}


\bibliographystyle{alpha}

\end{document}